\documentclass[11pt]{amsart}

\usepackage[margin=3.5cm]{geometry}
\usepackage{amssymb, amsmath}
\usepackage{amscd}
\usepackage[alwaysadjust]{paralist}
\usepackage{hyperref}
\usepackage[capitalize]{cleveref}
\usepackage{enumitem}
\usepackage{esint}
\usepackage{graphicx} 
\usepackage[rgb]{xcolor} 
\usepackage{verbatim}

\numberwithin{equation}{section}

\newtheorem{cor}[equation]{Corollary}
\newtheorem{lemma}[equation]{Lemma}
\newtheorem{prop}[equation]{Proposition}
\newtheorem{theorem}[equation]{Theorem}

\theoremstyle{definition}
\newtheorem{defn}[equation]{Definition}
\newtheorem{exa}[equation]{Example}

\newtheorem{rem}[equation]{Remark}

\newtheorem*{acknowledge}{Acknowledgments}

\def\IN{\mathbb N}
\def\IR{\mathbb R}

\def\IZ{\mathbb Z}

\newcommand{\Ric}{\operatorname{Ric}}
\newcommand{\Scal}{\operatorname{Scal}}
\newcommand{\genus}{\operatorname{genus}}

\newcommand{\diam}{\operatorname{diam}}
\newcommand{\dist}{\operatorname{dist}}
\newcommand{\loc}{\operatorname{loc}}

\newcommand{\rk}{\operatorname{rank}}

\newcommand{\cL}{\mathcal{L}}

\newcommand{\cS}{\mathcal{S}}

\newcommand{\sys}{\operatorname{sys}}

\newcommand{\area}{\operatorname{area}}
\newcommand{\length}{\operatorname{length}}
\newcommand{\ulsc}{\operatorname{ulsc}}

\newcommand{\IRP}{\mathbb{R}\mathrm{P}}

\newcommand{\R}{\mathbb{R}}

\begin{document}

\title[Large genus minimal surfaces in positive Ricci curvature]{The systole of large genus minimal surfaces in positive Ricci curvature}

\author{Henrik Matthiesen}
\address
{HM: Max Planck Institute for Mathematics,
Vivatsgasse 7, 53111 Bonn
\newline
current adress: Department of Mathematics, University of Chicago,
5734 S. University Ave, Chicago, Illinois 60637}
\email{hmatthiesen@math.uchicago.edu}
\author{Anna Siffert}
\address{Mathematisches Institut\\
Einsteinstr. 62\\
48149 M\" unster\\
Germany}
\email{ASiffert@uni-muenster.de}

\begin{abstract} 
We use Colding--Minicozzi lamination theory to show that the systole, and more generally any homology systole, of a sequence of embedded minimal surfaces in an ambient three-manifold of positive Ricci curvature tends to zero as the genus becomes unbounded.
\end{abstract}

\maketitle

\section{Introduction}

In 1985, Choi and Schoen \cite{CS} proved that the space of compact, embedded minimal surfaces with bounded genus in a closed ambient three-manifold $M$ of positive Ricci curvature 
is compact in the $C^\ell$-topology for any $\ell\geq 2$.
Conversely, in the present paper we study properties of minimal surfaces in
such ambient manifolds if the genus becomes unbounded.

\smallskip

Our main result shows that the systole of such a sequence of minimal surfaces tends to zero.
Recall that the \emph{systole} of a closed surface $\Sigma\subset M$ is defined to be
\begin{equation*} 
\sys(\Sigma):= \inf \{\length(c) \ \lvert \ c \colon S^1 \to \Sigma \ \text{non-contractible} \}.
\end{equation*}
Note that this takes into account all curves that do not bound a disk in $\Sigma$.
Similarly, the \emph{homology systole} is given by
\begin{equation*} 
\sys^h(\Sigma):= \inf \{\length(c) \ \lvert \ 0 \neq [c] \in H_1(\Sigma; \IZ/2\IZ) \},
\end{equation*}
taking into account only curves that are not a boundary in $\Sigma$.
Clearly, we have 
\begin{equation*}
\sys(\Sigma) \leq \sys^h(\Sigma).
\end{equation*}
More generally, for $k\in \IN^*$, let us define the \emph{$k$-th homology systole} by
\begin{equation*} 
\sys^h_k(\Sigma):= \inf \{ \max_{i=1,\dots,k} \length(c_i) \\ \lvert \ \rk(\langle c_1,\dots c_k \rangle ) = k \},
\end{equation*}
where the span $\langle c_1,\dots ,c_k \rangle$ is taken in $H_1(\Sigma;\IZ/2\IZ)$.\\
We use $\IZ/2\IZ$-coefficients here to deal with orientable and non-orientable surfaces simultaneously.
Of course, for orientable surfaces we can equivalently use $\IZ$-coefficients.

\smallskip

We can now state our main result.

\begin{theorem} \label{thm_intrinsic}
Assume that $(M,g)$ is a three-manifold with positive Ricci curvature.
Let $k \in \IN^*$ and consider a sequence $(\Sigma_j)_{j \in \IN}$ of closed, embedded minimal surfaces in $M$ with $\chi(\Sigma_j) \to  -\infty$ as $j \to \infty$.
Then the $k$-th homology systole satisfies
\begin{equation*}
\sys^h_k(\Sigma_j) \to 0,
\end{equation*}
as $j \to \infty$.
\end{theorem}

Thanks to the recent work
 of Chodosh--Mantoulidis on the Allen--Cahn equation \cite{CM18},
 any closed three-manifold with positive Ricci curvature contains a sequence of embedded minimal surfaces with unbounded genus,
 see also the related earlier work by Marques--Neves \cite{MN17} and Aiex \cite{aiex}.
Their construction gives a sequence of minimal surfaces $(\Sigma_p)_{p \in \IN}$ with 
$$
\area(\Sigma_p) \sim \genus(\Sigma_p)^{1/3} \sim p^{1/3},
$$
i.e.\ area growing sublinearly in the genus.
In fact, the same result has now also been established using Almgren--Pitts min-max theory through the works of Marques--Neves \cite{MN18} and Zhu \cite{zhu}.
At this point we also would like to point out Song's work settling the general case of Yau's conjecture \cite{S18}
and the papers by 
Irie--Marques--Neves \cite{MNI18}, Marques--Neves--Song \cite{MNS17}, and Liokumovich--Marques--Neves
\cite{LMN18} giving information on the distribution of min-max minimal hypersurfaces in the ambient manifold for generic metrics.
In a smiliar direction \cref{thm_intrinsic} provides information on embedded minimal surfaces of high complexity.
Because of the sublinear growth of the area, \cref{thm_intrinsic} is automatically true for min-max minimal surfaces thanks to general results on systoles of surfaces \cite{BPS} (cf.\ \cref{thm_sys_ratio_1}).
However, \cref{thm_intrinsic} applies to any family of minimal surfaces, not only those arising from min-max methods.

The best known bound for the area of embedded minimal surfaces in an ambient three manifold of positive Ricci curvature is
linear in the genus \cite{CW}.
More precisely, we have that
$$
\area(\Sigma) \leq C (\genus(\Sigma)+1)
$$
for a constant $C$ depending on the topology of $M$ and the lower bound on the Ricci curvature.
It is by no means clear if this bound is sharp and \cref{thm_intrinsic} could be considered as some hint towards the non-sharpness of the linear bound.
It appears to be an interesting question to understand the maximal possible area growth of a sequence of embedded, minimal surfaces with genus tending to infinity.
To the best knowledge of the authors, there is currently no family of closed, embedded minimal surfaces in $\mathbb{S}^3$ (or any closed three-manifold with positive Ricci curvature) known
with unbounded genus such that $\genus(\Sigma)=o(\area(\Sigma)^2)$ as $\genus(\Sigma) \to \infty$.

An interesting family of minimal surfaces (potentially immersed and of higher codimension) with linear growth of the area in the genus was recently constructed by the authors \cite{MS3} (also relying on earlier work by Petrides \cite{petrides}) as maximizing metrics for the first Laplace eigenvalue normalized by area.
These are immersed minimal surfaces in $\mathbb{S}^N$ for $N \leq C(\genus(\Sigma))$.
There is not much known about the geometry of these surfaces.
In \cite{MS3} we proved that infinitely many of them do not arise as branched covers over $\mathbb{S}^2$.
But it is for instance unclear if an infinite number of these surfaces arises as embedded minimal surfaces in $\mathbb{S}^3$.
One way to rule this out would be to show that $\area(\Sigma) = o (\genus(\Sigma))$ as $\genus(\Sigma) \to \infty$
for closed, embedded minimal surfaces in $\mathbb{S}^3$.

It also worth emphasizing that \cref{thm_intrinsic} crucially relies on the surface being embedded, since high-degree covers of a given minimal surface of positive genus provide trivial counterexamples to the immersed version.

We want to briefly discuss why our result is more subtle than one might expect at first glance.
In general, one could expect that $\sys(S_i) \to 0$ for any (i.e.\ not necessarily minimal) sequence of surfaces $S_i$ in $S^3$ with $\genus(S_i) \to \infty$ at least as long as $S_i$ are unknotted.
However, one can easily produce counterexamples to this using the Nash--Kuiper theorem:
Take a surface $S_\gamma$ of genus $\gamma$ with systole $\sys(S_\gamma)\geq c_0>0$.
By the Nash--Kuiper theorem, there is a $C^{1,\alpha}$-isometric embedding of $S_\gamma$ in an arbitrarily small ball $B_\delta \subset \IR^3$.
After smoothing this and applying stereographic projection, we get a sequence of closed, unknotted surfaces of unbounded genus in $S^3$, which have systole uniformly bounded from below.

Moreover, \cref{thm_intrinsic} does not hold without any assumptions on the ambient geometry.

\begin{exa} \label{example}
Denote by $\Sigma_\gamma$ a closed surface of genus $\gamma$ for $\gamma \geq 2$.
It is shown in \cite{To69} (see also \cite{Ne76} for a generalization) that the three-manifold $M=S^1 \times \Sigma_\gamma$ admits fibre bundles
\begin{equation} \label{fibration}
\Sigma_\delta \to M \to S^1
\end{equation} 
for $\delta=\gamma+n(\gamma-1)$ and $n \in \IN$.
Since $\pi_2(S^1)=0$, the long exact sequence for homotopy groups associated to these fibrations implies that $\Sigma_\delta \to M$ is incompressible,
i.e.\ the induced map $\pi_1(\Sigma_\delta) \to \pi_1(M)$ is injective.
It follows from \cite[Theorem 3.1]{SY79} that there are immersed minimal surfaces $S_\delta$ in $M$ which are diffeomorphic to $\Sigma_\delta$ and the induced map on $\pi_1$ is given by the inclusion of the fibres from \eqref{fibration}.
Moreover, \cite[Theorem 5.1]{FHS83} implies that these are not only immersions but even embeddings.
Since $\pi_1(S_\delta) \to \pi_1(M)$ is injective, we have in particular that
\begin{equation*}
\sys(S_\delta) \geq \sys(M) >0.
\end{equation*}
On the other hand, it follows from \cite[Theorem 5.2]{SY79} that $M$ does not admit any metric of positive scalar curvature.
\end{exa}

\subsection*{Main problems and strategy} 
Let us for simplicity focus on the case of $M$ being simply connected, $k=1$ and the systole instead of the homology systole.
We want to argue by contradiction and consider a sequence of minimal surfaces $\Sigma_j \subset M$ with $\sys(\Sigma_j) \geq l_0 >0$ and $\genus(\Sigma_j) \to \infty$. 
In general, we would like to pass to a limit $\Sigma_j \to \cL$ in the class of minimal laminations and argue that $\cL$ has a stable leaf, which would easily lead to a contradiction since $M$ has positive Ricci curvature.
The problem about this is that we can only do this outside the closed set at which $|A^{\Sigma_j}|^2$ blows-up.
A priori, the blow-up set could even be all of $M$.
Work of Colding and Minicozzi gives strong structural information about the blow-up set if the surfaces in question have bounded genus.
The main step of our proof is to show that the sequence $\Sigma_j$ as above can locally be dealt with in this framework.
The reason why this is not obvious is that  we do not have $-\Delta_{\Sigma_j} d^2(x,\cdot) \leq 0$ globally (as it is the case for minimal surfaces in $\IR^3$).
Therefore, the assumption on $\sys(\Sigma_j)$ does not directly imply that there is $R_0=R_0(l_0)$ such that the intrinsic balls $B^{\Sigma_j}(x,R_0)$ are contained in (intrinsic) \emph{disks} in the intersection $B(x,R_0) \cap \Sigma_j$ with an extrinsic ball.
Instead, $B^{\Sigma_j}(x,R_0)$ is contained in some disk $D_x^j \subset \Sigma_j$ but $D_x^j$ could leave any mean convex ball $B(x,r)$ centered at $x$.
The main step is to show that this is impossible after going to a (potentially much) smaller scale.
The proof of this is of global nature and also relies on the positivity of the Ricci curvature of $M$.
It also proceeds by contradiction and follows broadly the same strategy.
Given a contradicting sequence we try to find a stable minimal surface in $M$.
The key step to achieve this is to show that $\Sigma_j$ serves as a good barrier for a minimization problem in $M$.
This in turn is shown by promoting information about singularities of $\Sigma_j$ for $j \to \infty$ across scales using the maximum principle.

The general case of the theorem follows similar steps but is technically more involved.
This requires for instances a 
 more careful blow-up argument
 in the case $k \geq 2$ and also makes use of
 some additional elementary topological arguments.

\subsection*{Organization}
In \cref{sec_cml} we provide necessary background from \cite{CM5} on Colding--Minicozzi lamination theory of minimal surfaces with some control on the topology.
\cref{chord_arc} contains two weak chord-arc properties for minimal surfaces contained in small extrinsic balls of an ambient three-manifold.
In \cref{prelim_surf} we give some rather elementary preliminaries on surfaces and topology and recall a fundamental result from systolic geometry which are needed to prove
our main result, \cref{thm_intrinsic} first in the case $k=1$ in \cref{main} and in the general case in \cref{main_2}

\begin{acknowledge} 
 Both authors would like to thank the Max Planck Institute for Mathematics in Bonn for support and excellent working conditions.
 We would also like to thank Fabian Henneke for explanations related to \cref{example} and Andr{\'e} Neves for asking us, if it is possible to find more than a single short curve. 
\end{acknowledge}

\section{Background on Colding--Minicozzi lamination theory} \label{sec_cml}
Colding and Minicozzi developed a theory that describes how minimal surfaces of uniformly bounded genus in 
an ambient three-manifold can degenerate in the absence of curvature bounds.
We use this section to provide a very brief introduction to those parts of their theory that will be relevant in the present paper.
We will focus here on the case of planar domains, since this is sufficient for our purposes.

\smallskip

We start by recalling the definition of a lamination.

\begin{defn}[see Appendix B in \cite{CM4}]
\begin{enumerate}
\item A codimension one lamination on a three-manifold $M$ is a collection $\mathcal{L}$ of
smooth disjoint surfaces $\Gamma\subset M$, the so-called leaves, such that $\cup_{\Gamma\in\mathcal{L}}\Gamma$ is closed.
Furthermore, for each point $x\in M$, there exists an open neighborhood $U$ of $x$ and a coordinate chart, $(U,\Phi)$,
with $\Phi(U)\subset\R^3$ so that in these coordinates the leaves in $\mathcal{L}$ pass through $\Phi(U)$ in slices
of the form $(\R^2 \times\{t\})\cap\Phi(U).$
\item A foliation is a lamination for which $M=\cup_{\Gamma\in\mathcal{L}}\Gamma$, i.e. the union of the leaves is all of $M$.
\item A minimal lamination is a lamination whose leaves are minimal. 
\item A Lipschitz lamination is a lamination for which the chart maps $\Phi$ are Lipschitz. 
\end{enumerate}
\end{defn}

Given any sequence of minimal surfaces $\Sigma_j \subset M$, we consider the \emph{singular} or \emph{blow-up set}
 \begin{equation*}
\cS=\{z \in M \ \lvert\  \inf_{\delta>0} \sup_j \sup_{B(z,\delta)} |A^{\Sigma_j}|=\infty\},
\end{equation*}
i.e.\ the points $z$ where the curvature blows up.
Up to taking a subsequence one can always pass to a limit
\begin{equation*}
\Sigma_j \to \cL \ \text{in} \ M \setminus \cS,
\end{equation*}
where the convergence is in $C^{0,\alpha}$ and the limit lamination is a minimal Lipschitz lamination.

\smallskip

In the case of minimal surfaces $\Sigma_j \subset B(0,R_j) \subset \IR^3$ with bounded genus and $\partial \Sigma_j \subset \partial B(0,R_j)$ one can always extract a subsequence such that either $R_j \to \infty$ or with $R_j$ bounded.
In the former case one can reach much stronger conclusions on the structure of the limit lamination, 
see e.g.\ the example in \cite{CM04e}. 
Since we only deal with the local case, i.e.\ $R_j=R$ is fixed, which in general only allows to draw significantly weaker conclusions about the structure of the limit lamination,
we do not discuss stronger conclusion valid in the global case.

\smallskip

We first consider the case of $\Sigma_j$ being disks.
Colding and Minicozzi proved \cite{CM1,CM2,CM3,CM4} that
every embedded minimal disk is either a graph
of a function or is a double spiral staircase where each staircase is a multivalued
graph. 
More precisely, they show that if the curvature is large at some
point (and thus the surface is not a graph), then the surface is a double spiral staircase
like the helicoid.

Below we also want to deal with the case where $\Sigma_j$ are more general 
domains than disks, namely, so-called uniformly locally simply connected (in short: ULSC) 
domains.

A sequence of minimal surfaces $\Sigma_j \subset M$ is called \emph{uniformly locally simply connected}
\footnote{We remark that this is stronger than the definition of Colding--Minicozzi in the case of non-planar domains.}
if given any compact $K \subset M$ there is
some $r>0$  such that 
\begin{equation*}
\Sigma_j \cap B(x,r) \ \text{consists of disks for any $x \in K$}.
\end{equation*}
Moreover, we define 
\begin{equation*}
\cS_{\ulsc} :=\{z\ \in \cS \ :  \ \Sigma_j \ \text{is ULSC near}\ z \}.
\end{equation*}

The main local structural result we need for (not necessarily globally planar or bounded genus) ULSC sequences concerns so-called collapsed leaves, whose existence is described in the next lemma.
We assume that $\Sigma_j \to \cL'$ in $M \setminus \cS$, where $\Sigma_j$ is a ULSC sequence.

\begin{lemma}[Lemma II.2.3. in \cite{CM5}]\label{reg}
Given a point $x\in\cS=\cS_{\ulsc}$, there exists $r_0>0$ so that $B(x,r_0)\cap\mathcal{L}^{'}$ has a component $\Gamma_x$ whose closure $\overline{\Gamma_x}$ is a smooth minimal graph containing x and with boundary in $\partial B(x,r_0)$ (so $x$ is a removable singularity for $\Gamma_x$).
\end{lemma}

We want to emphasize that while \cite{CM5} starting at the end of Section II.1 makes the general assumption to be in the global case $R_j \to \infty$ this does not apply to everything contained in the following sections.
In particular a look at the proof of Lemma II.2.3 show that this does not make use of this assumption.
Similarly, an inspection of the arguments shows the statements from \cref{prop_collapsed} below are valid without this assumption.

The leaves of the limit foliation $\cL^{'}$ may not be complete. 
A special type of incomplete leaves are collapsed leaves.
A leaf $\Gamma$ of $\cL^{'}$ is \textit{collapsed} if there exists some $x\in\cS_{\ulsc}$ so that $\Gamma$ contains the local leaf $\Gamma_x$ given by Lemma\,\ref{reg}; see Definition II.2.9 in \cite{CM5}.

We now assume that the ambient manifold is given as $M=\bar M \setminus \{x_1,\dots,x_k\}$,
where $\bar M$ is complete and $x_i \in \bar M$.
In order to state the key structural results on collapsed leave we need to introduce some notation.
Given a leave $\Gamma \subset \mathcal{L}'$ we fix a point $x \in \cL'$ and write
$$
\Gamma_{clos}= \bigcup_{R>0} \overline{B^\Sigma(x,R)},
$$
where the closure is taken in $M$.

\begin{prop}[see Section II.3. in \cite{CM5}] \label{prop_collapsed}
Each collapsed leaf $\Gamma$ of $\cL'$ has the following
 properties:
\begin{enumerate}
\item  Given any $y \in \Gamma_{clos} \cap
\cS_{\ulsc}$, there exists $r_0 > 0$ so that the closure in $M$ of
each component of $\Gamma\cap B(y,r_0)$ is a compact embedded
 disk  with boundary in $\partial B(y,r_0)$.
 \\
Furthermore, $\Gamma\cap B(y,r_0)$ must contain the component
$\Gamma_y$ given by Lemma\,\ref{reg} and $\Gamma_y$ is the only
component of $\Gamma\cap B(y,r_0)$ with $y$  in its closure.
\item $\Gamma$ is a limit leaf.
\item $\Gamma$ extends to a complete minimal surface away from $\{x_1,\dots,x_k\}$
\footnote{i.e.\ there is $\Gamma'$ containing $\Gamma$ such that if a geodesic in $\Gamma'$ can not be extended it limits to some $x_i$}.
\end{enumerate}
\end{prop}

The sequences $\Sigma_j$ appearing in this manuscript will essentially all  be ULSC. 
This is equivalent to the fact that the singular set $\cS$ is given by $\cS_{\ulsc}$, i.e. $\cS=\cS_{\ulsc}$.
Although we will not directly apply the results for non-ULSC surfaces here, some of our arguments (in particular the proof of \cref{lem_no_necks}) are inspired by those in \cite{CM5} for this case.

\section{Chord arc properties}
 \label{chord_arc}
 
We need two weak chord-arc properties for minimal surfaces contained in small extrinsic balls of an ambient three-manifold.
Given $x \in M$ and $r>0$, we write $B(x,r)$ for the metric ball in $(M,g)$.
If $z \in \Sigma$ and $r>0$, we denote by $B^\Sigma(z,r)$ the metric ball of radius $r$ in $\Sigma$ with respect to the induced Riemannian metric.

\smallskip

Let $(M,g)$ be a closed Riemannian three-manifold. 
For $R_0>0$ sufficiently small, we consider minimal embedded disks $\Sigma$ in $B(x_0,R_0)$ for some $x_0 \in M$.
By $\Sigma_{x_0,r}$ we denote the connected component of $\Sigma \cap B(x_0,r)$ that contains $x_0$.

\begin{theorem} \label{thm_chord_arc}
Let $\Sigma \subset B$ be an embedded minimal disk with $x_0 \in \Sigma$.
There is $\alpha>0$ such that if $B^\Sigma (x_0,R) \subset \Sigma \setminus \partial \Sigma$, then $\Sigma_{x_0,\alpha R} \subset B^\Sigma(x_0,R/2)$.
\end{theorem}

This is proved in \cite{CM08} for minimal disks in $\IR^3$, in which case minimal surfaces have non-positive curvature.

The proof of \cref{thm_chord_arc} is exactly as the proof of \cite[Proposition 1.1]{CM08}.
This does not use that intrinsic subballs $B^\Sigma(x,R) \subset \Sigma$ of a minimal disk $\Sigma$ are disks again, but only that they are contained in disks and that $\Sigma_{x,r}$ is a disk provided that $\partial \Sigma \cap B(x,r)=\emptyset$.

\smallskip

We also need a related chord-arc property for uniformly locally simply connected surfaces.

\begin{theorem} \label{thm_chord_arc_2}
Let $\Sigma \subset B(x,R)$ be a minimal surface with $x \in \Sigma$.
Assume that there is $r>0$, such that $\Sigma \cap B(y,r)$ consists only of disks for any $y \in B(x,R-r)$.
Then, given $k \in \IN$ such that $kr \leq R$ there is $\beta_k>0$ such that if $B^\Sigma(x,\beta_k r) \cap \partial \Sigma = \emptyset$,
then 
$\partial (\Sigma_{x,kr}) \subseteq \partial B(x,kr)$.
\end{theorem}

This is stated in \cite[Appendix B.1]{CM5} for intrinsic instead of extrinsic balls.
In our setting, intrinsic balls that are contained in a disk may not be disks themselves.
The version stated above is proved as in \cite{CM5} with some easy changes using \cref{thm_chord_arc}.

\section{Some preliminaries on surfaces and topology}\label{prelim_surf}

In this section we recall some elementary and well known facts about the topology of surfaces.\
We also recall some results from systolic geometry.

\subsection{A result from systolic geometry}

We will use the following result from systolic geometry, that relates the area and the $k$-th homology systole.

\begin{theorem}[{\cite[Theorem 1.2]{BPS}, see also \cite{Gr}}] \label{thm_sys_ratio_1}
Let $\eta \colon \IN \to \IN$ be a function such that
\begin{equation*}
\lambda:= \sup_\gamma \frac{\eta(\gamma)}{\gamma}<1.
\end{equation*}
Then there exists a constant $C_\lambda$ such that for every closed, orientable Riemannian surface $\Sigma$ of genus $\gamma$, we have
\begin{equation*}
\sys^h_{\eta(\gamma)}(\Sigma) \leq C_\lambda \frac{\log(\gamma+1)}{\sqrt{\gamma}}\sqrt{\area(\Sigma)}.
\end{equation*}
\end{theorem}

Recall that a non-orientable surface $\Sigma$ can be written as a connected sum $\Sigma=\Sigma_1 \# \Sigma_2$, with $\Sigma_1$ closed, orientable and
 $\Sigma_2$ diffeomorphic to $\IRP^2$ or $\IRP^2 \#\IRP^2$.
If we replace $\Sigma_2$ by a disk, \cref{thm_sys_ratio_1} easily implies the following for non-orientable surfaces.

\begin{cor}
Let $\eta$ and $\lambda$ be as above, then there is a constant $C_\lambda$, such that for every closed, non-orientable surface of non-orientable genus $\delta$, we have
\begin{equation*}
\sys^h_{\eta(\gamma_\delta)}(\Sigma) \leq  C_\lambda \frac{\log(\gamma_\delta+1)}{\sqrt{\gamma_\delta}}\sqrt{\area(\Sigma)},
\end{equation*}
where $\gamma_\delta=\lfloor (\delta-1)/2 \rfloor$.
\end{cor}

We will only use the following consequence of these results.

\begin{cor}  \label{thm_sys_ratio}
Let $(\Sigma_j)$ be a sequence of surfaces with $-\chi(\Sigma_j)\to \infty$.
If $\area(\Sigma_j)=O((-\chi(\Sigma_j))^{\alpha})$ for some $0\leq \alpha<1$, then, for any $k \in \IN$, we have 
\begin{equation*}
\sys^h_k(\Sigma_j) \to 0
\end{equation*}
 as $j \to \infty$.
\end{cor}

To put this into context, notice that the Choi--Wang bound \cite{CW} implies
 for a closed, embedded, orientable, minimal surface $\Sigma$ that
\begin{equation*}
\area(\Sigma) \leq C (\genus(\Sigma)+1),
\end{equation*}
where $C=C(k)$, if $\Ric(M) \geq k >0$.

\subsection{Some elementary facts about the topology of surfaces}

\begin{lemma} \label{lem_curve_sep}
Let $\Sigma$ be a closed surface and $c \subset \Sigma$ a simple closed curve.
Then $[c]\neq 0 \in H_1(\Sigma ; \IZ/2\IZ)$ if and only if $c$ is non-separating.
\end{lemma}

\begin{proof}
Clearly, if $c$ is separating, then $[c]=0$ in $H_1(\Sigma, \IZ/2\IZ)$.
On the other hand, if $c$ is non-separating, there is a curve $d$ such that $|c \cap d|=1$.
In particular, from the intersection pairing, $[c] \neq 0 \in H_1(\Sigma, \IZ/2\IZ)$.
\end{proof}

\begin{lemma} \label{lem_top_sep}
Let $\Sigma$ be a closed surface and $c \subset \Sigma$ a simple closed curve, such that $[c]\neq 0 \in H_1(\Sigma;\IZ/2\IZ)$.
If we have $c \subset B(x,r)$ and a simple closed curve $d \subset \Sigma$ homologous to $c$ with $d \subset M \setminus B(x,r)$, then $d$ is separating in $\Sigma \cap M \setminus B(x,r)$.
\end{lemma}

\begin{proof}
If $d$ is non-separating in $\Sigma \setminus B(x,r)$ we can find a curve $e \subset \Sigma \setminus B(x,r)$ that intersects $d$ exactly once.
On the other hand, $c \cap e = \emptyset$, which is impossible, since $c$ and $d$ are homologous. 
\end{proof}

\begin{defn}
Let $\Sigma \subset M$ be an embedded surface, $x \in M$ and $r>0$.
We say that $c \colon S^1 \to \Sigma$ is \emph{contractible on scale $r$ at $x$} if there is a disk $\phi \colon D \to B(x,r) \cap \Sigma$ with $\left. \phi \right|_{S^1} = c$.
If there is some $x \in M$ such that $c$ is contractible on scale $r$ at $x$ we say that \emph{$c$ is contractible on scale $r$}.
\end{defn}

We will make frequent use of the following version of the maximum principle related to this definition. 
Choose $r>0$ such that any ball $B(x,s) \subset M$ with $s \leq r$ and $x \in M$ is mean convex.
If $\Sigma \subset M$ is a complete minimal surface and $c \subset \Sigma$ is contractible on scale $r$ (at $x$) then $c$ is contractible on scale $s$ (at $x$) for any $s \leq r$.

\begin{lemma} \label{lem_gen_fund}
Let $\Sigma \subset M$ be a surface, $x \in \Sigma$, and $R>0$, 
If any curve $d \colon S^1 \to B^\Sigma(x,R)$ with $\length(d) \leq 3R$ is contractible on scale $r$ at $x$
then any curve $c \colon S^1 \to B^\Sigma(x,R)$ is contractible on scale $r$ at $x$.
\end{lemma}

\begin{proof}
Let $c \colon S^1 \to B^\Sigma(x,R)$ be a loop. 
Choose a subdivision
\begin{equation*}
0=t_0 < t_1 < \dots < t_{k-1}<t_k=1,
\end{equation*}
 of $[0,1]$ such that 
\begin{equation*}
\length( c_{|[t_i,t_{i+1}]}) \leq R.
\end{equation*}
Fix curves $d_i \colon I \to B^\Sigma(x,R)$ with $d_i(0)=x$ and $d_i(1)=c(t_i)$ and such that 
\begin{equation*}
\length(d_i) \leq R.
\end{equation*}
We can then write
\begin{equation*}
\begin{split}
c= &
 \left( c_{|{[t_{k-1},t_k]}} \ast d_{k-1} \right) \ast \left(\bar{d}_{k-1} \ast c_{|{[t_{k-2},t_{k-1}]}} \ast d_{k-2} \right) \ast  \\
 &\dots \ast \left(\bar{d_2} \ast c_{|{[t_1,t_2]}} \ast d_1 \right) \ast \left( \bar{d_1} \ast c_{|{[t_0,t_1]}}\right),
\end{split}
\end{equation*}
which implies the assertion.
\end{proof}

Since the Hurewicz homomorphism $\pi_1(B^\Sigma(x,R),x) \to H_1(B^\Sigma(x,R);\IZ)$ as well as the map $H_1(B^\Sigma(x,R);\IZ) \to H_1(B^\Sigma(x,R);\IZ/2\IZ)$ are surjective, we immediately get the following corollary.

 \begin{cor} \label{cor_gen_hom}
Let $\Sigma$ be a surface, $x \in \Sigma$, and $R>0$, then the group $H_1(B^\Sigma(x,R);\IZ/2\IZ)$ is generated by curves of length at most $3R$.
\end{cor}

\begin{lemma} \label{lem_cov_sep}
Let $\Sigma$ be a closed surface and $\pi \colon \hat \Sigma \to \Sigma$ a covering.
Consider a simple closed curve $c \subset \Sigma$ and its preimage $\hat c = \pi^{-1}( c) \subset \hat \Sigma$.
If  $c$ is separating, then also $\hat c$ is separating.
\end{lemma}

\begin{proof}
If $c$ is separating, we can write $\Sigma \setminus c = \Sigma_+ \cup \Sigma_-$ with connected surfaces $\Sigma_\pm$.
Moreover, there is a function $f \colon \Sigma \to [-1,1]$ such that $\{f=0\}=c$ and $\Sigma_{\pm}=\{f \gtrless 0\}.$
We can then consider the lifted function $\hat f = f \circ \pi$, which clearly satisfies $\{ \hat f =0\}=\hat c$.
Therefore, $\hat c$ separates $\hat \Sigma$ into $\hat \Sigma_-=\{\hat f <0\}$ and $\hat \Sigma_+=\{ \hat f >0 \}$.
\end{proof}

It will be important to keep in mind that the domains $\hat \Sigma_{\pm}$ might be disconnected and $\hat c$ is potentially not the boundary of a compact subsurface.

\section{Existence of one short curve}\label{main}

Throughout this section let $(M,g)$ be a closed three-manifold with positive Ricci curvature.
In order to prove \cref{thm_intrinsic}, we want to argue by contradiction.
Therefore, we study properties of a sequence $\Sigma_j \subset (M,g)$ of closed, embedded minimal surfaces
 with $\sys^h_k(\Sigma_j) \geq l_0>0$.
 More precisely, we will be concerned with a limit lamination
 \begin{equation*}
 \Sigma_j \to \cL \ \text{in} \ M \setminus \cS
 \end{equation*}
 of such a sequence.
 For the sake of clarity, and since we need the corresponding arguments anyways, we will focus first on the case $k=1$, i.e.\ the first homology systole, and explain the necessary extensions to handle
 the general case afterwards.
 
\subsection{The singular set is non-empty}\label{sing_set}
 
We start with a simple observation concerning the maximum of the curvature of a sequence of minimal surfaces in $M$ with unbounded genus.
It says, that for a sequence of minimal surfaces of unbounded genus $\Sigma_j \subset M$, we necessarily have $\cS \neq \emptyset$.
This works without any assumption on the systole.

\begin{lemma} \label{lem_curv}
Let $\Sigma_j \subset (M,g)$ be a sequence of closed, embedded minimal surfaces with $\chi(\Sigma_j) \to -\infty$.
Then there is a sequence of points $z_j \in \Sigma_j$ such that $|A^{\Sigma_j}|^2(z_j) \to \infty.$
\end{lemma}

\begin{proof}
Assume that there is a constant $C>0$, such that 
\begin{equation} \label{pt_curv_bd}
\sup_j \sup_{\Sigma_j}|A^{\Sigma_j}|^2\leq C.
\end{equation}
By scaling we may for simplicity assume that $|\sec(M)| \leq 1$.
Thus, by minimality and the theorem of Gau{\ss}--Bonnet, the total curvature satisfies
\begin{equation} \label{lem_curv_1}
\begin{split}
\int_{\Sigma_j} |A^{\Sigma_j}|^2 d\mu_{\Sigma_j} 
&=
-2\int_{\Sigma_j} (K^{\Sigma_j}-\sec(T_x \Sigma_j)) d\mu_{\Sigma_j}(x)
\\ 
& \geq
4\pi |\chi(\Sigma_j)|-2 \area(\Sigma_j).
\end{split}
\end{equation} 
On the other hand we have 
\begin{equation} \label{lem_curv_2}
\int_{\Sigma_j} |A^{\Sigma_j}|^2 d\mu_{\Sigma_j} \leq C \area(\Sigma_j)
\end{equation}
by assumption.
Combining \eqref{lem_curv_1} and \eqref{lem_curv_2}, we obtain
\begin{equation*}
4 \pi |\chi(\Sigma_j)| \leq (C+2) \area(\Sigma_j).
\end{equation*}
By assumption the left hand side tends to infinity, therefore we find that
\begin{equation*}
\area(\Sigma_j)\to \infty
\end{equation*}
as $j \to \infty$.

\smallskip
We consider the universal covering $\pi \colon \tilde M \to M$, where $\tilde M$ is compact by the Bonnet-Myers theorem.
Clearly, the minimal surfaces 
$$
\hat \Sigma_j := \pi^{-1}(\Sigma_j)
$$ 
also satisfy the pointwise curvature bound \eqref{pt_curv_bd}
and have diverging area,
\begin{equation} \label{eq_area_diverge}
\area(\hat \Sigma_j) \to \infty.
\end{equation}
The pointwise curvature bound \eqref{pt_curv_bd} allows us to pass to a subsequence (not relabeled) such that
\begin{equation*}
\hat \Sigma_j \to \cL \ \text{in} \ C^{0,\alpha}(\tilde M),
\end{equation*} 
where $\cL$ is a Lipschitz lamination, whose leaves are smooth, complete minimal surfaces.
Moreover, since $\area(\Sigma_j) \to \infty$, 
we can conclude that there needs to be at least one leaf
$\Gamma$ with stable universal cover, which also implies that $\Gamma$ is compact, hence diffeomorphic to $S^2$ thanks to \cite{FS} and \cite{SY83}.
For the convenience of the reader we include the argument here following the proof of \cite[Theorem 1.3]{CKM}.

By passing to another subsequence and using \eqref{eq_area_diverge} we find that there has to be a point $p \in \Gamma$ such that
$$
\liminf_{j \to \infty} \area(\hat \Sigma_j \cap B(p,r)) \to \infty
$$
for any $r>0$.
Since $\hat \Sigma_j$ are embedded and by the curvature bound \eqref{pt_curv_bd} this implies that for $j$ sufficiently large and $r>0$ sufficiently small (but only depending on the ambient geometry and the curvature bound),
we have that $\hat \Sigma_j \cap B(p,r)$ is given as the union of $n(j) \to + \infty$
graphical components over $\Gamma \cap B(p,r)$.
Let $U \subset \tilde \Gamma$ be bounded and simply connected with $\tilde p \in U$, where $\tilde p$ projects to $p$.
Using the curvature bound, a covering argument and the standard elliptic theory we find that for $j$ sufficiently large we can find at least two functions
$v_{1,j},v_{2,j}$ (out of the lifts of those $n(j)$-many above) defined on $U$ such that the graphs define disjoint minimal surfaces over $U$ and
$\inf |v_{2,j}-v_{1,j}| \to 0$.
Using the Harnack inequality we find that $w_j =  \inf |v_{2,j}-v_{1,j}|^{-1} (v_{2,j}-v_{1,j})$ converges to a non-trivial signed solution of the Jacobi equation, hence $U$ is stable.
Since this applies to any such $U$ it follows that $\tilde \Gamma$ is stable.

It follows from \cite{FS} and \cite{SY83}, that for any disk $D \subset \tilde \Gamma$ and any $z \in D$ we have that
$$
d(x,\partial D) \leq \frac{2 \pi \sqrt{2}}{\sqrt{3 \kappa_0}},
$$
where $\Scal\geq\kappa_0>0$ on $D$. 
Since this applies to any such disk it follows that $\tilde \Gamma$ is compact, hence a sphere.
Since $\tilde M$ is simply connected, it does not contain any embedded real projective plane.
Therefore, we need to have $\tilde \Gamma = \Gamma$.
In particular, $\Gamma$ is a closed, two sided, stable minimal surface in $\tilde M$, which gives the desired contradiction.
\end{proof}

\begin{rem} \label{rem_sys}
Under the additional assumption that $\sys^h(\Sigma_j) \geq l_0>0$ we could have used \cref{thm_sys_ratio} instead of the theorem of Gau{\ss}-Bonnet to obtain that $\area(\Sigma_j)$ has to be unbounded.
However, this relies on the assumption on the systole and is less elementary.
We will exploit such kind of argument below in the proof of the existence of multiple pinching curves.
\end{rem}

\subsection{Localized systole and contractibility radius I}

We now start to aim for \cref{thm_intrinsic} for $k=1$, i.e.\ we show that there is at least one homologically non-trivial curve that becomes arbitrarily short.
By \cref{lem_curv}, in order to prove \cref{thm_intrinsic} using a contradiction argument invoking a limit lamination, we are forced to study the structure of a limit lamination of $(\Sigma_j)_{j \in \IN}$ in the presence of a non-empty singular set.
In this subsection we use the global positivity of the Ricci curvature to rule out rather general neck-pinch singularities under appropriate assumptions.

\smallskip

We now fix $r_0>0$ sufficiently small such that firstly the results from \cref{chord_arc} apply
in any ball $B(x,r_0)$ and, secondly, all balls $B(x,r) \subset M$ with $r \leq r_0$ have strictly mean convex boundary.

For an embedded closed surface $\Sigma \subset M$ and a point $x \in M$ we write
$$
C(x,r) = C^\Sigma(x,r) = \{ c \colon S^1 \to \Sigma \cap B(x,r)  \ : \ 0 \neq [c] \in \pi_1(\Sigma \cap B(x,r)) \}.
$$
Note that $c \in C(x,r)$ could still be globally contractible in $\Sigma$.
We also write
$$
C(x) = C^\Sigma (x) = C(x,r_0)
$$

\begin{defn}
We call
$$
c(\Sigma) = 
\inf_{x \in M} \sup\{ r > 0 \ : \pi_1(\Sigma \cap B(x,r),x)=0\} 
$$
the \emph{contractibility radius} of $\Sigma$
and
$$
\sys_r(\Sigma)=
\inf_{x \in M} \inf_{c \in C(x,r)} \length(c)
$$
the \emph{$r$-local systole} of $\Sigma$.
\end{defn}

While the two definitions seem closely related, there is an important difference to be pointed out.
Note that both of these are defined by looking at the intersection of $\Sigma$ with extrinsic balls.
However, the $r$-local systole still refers to intrinsic distances, i.e.\ we only localize extrinsically.
Also observe that we have
$$
\sys_r(\Sigma) = \inf_{x \in \Sigma} \sys(\Sigma \cap B(x,r)).
$$

Note that both the contractibility radius and the $r$-local systole refer to extrinsic balls of radius $r$.
One of the main challenge of our arguments is that in general there might be no strong connection between the systole (which is intrinsic) and the extrinsic $r$-local systole as indicated in \cref{fig_sys}.

\begin{figure}[ht] 
	\centering
  \includegraphics[width=6cm]{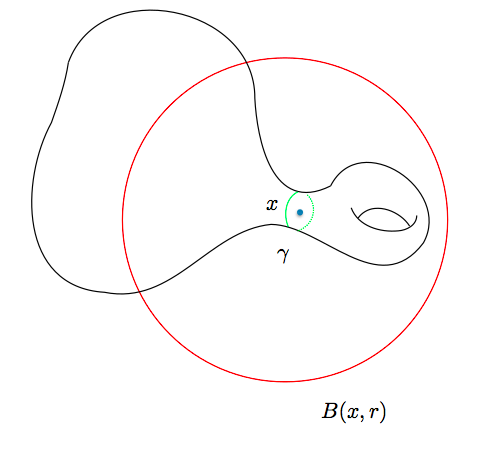}
	\caption{A surface with small $r$-local systole but not too small systole as the curve $\gamma$ is globally contractible.}
	\label{fig_sys}
\end{figure}

The main goal of this subsection is to show that the control on the homology systole gives control on the contractibility radius.
As an intermediate step we first obtain control on the localized systole in the next two lemmata.

The next lemma is our key scale breaking argument indicated in \cref{fig_max_princ}.
Via the maximum principle we transfer some connectedness properties from the scale of certain singularities of a limit lamination to a definite scale.
Given a very short and separating curve we show that both connected components of the complement have to extend a definite amount away.

\begin{lemma} \label{lem_good_barrier}
Let $\Sigma \subset M$ be a closed minimal surface such that
$
\sys^h(\Sigma) \geq l_0.
$
There is $l_1=l_1(M,l_0) \leq \min(r_0,l_0/2)$ 
with the following property.
Suppose that
$
\sys_{r_0}(\Sigma) \leq l_1
$
and that $c \in C(x)$ is a simple closed curve for some $x \in M$ such that
$$
\length(c) \leq 2 \sys_{r_0}(\Sigma).
$$
Then $\Sigma \setminus c$ has two connected components $\Sigma_1$ and $\Sigma_2$ and these satisfy
\begin{equation*}
\Sigma_i \cap \partial B(x,r_0) \neq \emptyset
\end{equation*}
for $i=1,2$.
\end{lemma}

\begin{figure}[ht]
	\centering
  \includegraphics[width=10cm]{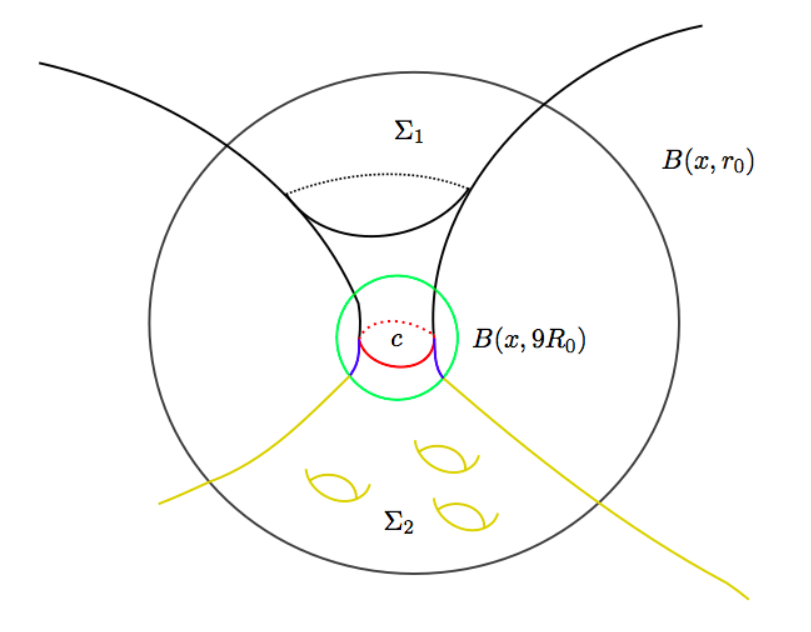}
	\caption{The proof of \cref{lem_good_barrier}: Because $B(x,r)$ is mean convex for any $r \leq r_0$ at least one component of $\Sigma \setminus c$ has to leave $B(x,r_0)$, say $\Sigma_1$. Once we can show that $\Sigma_2$ is forced to leave $B(x,9R_0)$ (still on the scale of $c$), the maximum principle gets us all the way to $\partial B(x,r_0)$.
	}
	\label{fig_max_princ}
\end{figure}

\begin{proof}
Write $R_0=\length(c)/8$ and assume that $R_0 \leq \min(r_0/16, l_0/16)$.
Clearly, by the choice of $c$ this implies that $\Sigma \setminus c$ has two connected components, denoted $\Sigma_1$ and $\Sigma_2$
with $\partial \Sigma_i = c$.
Also note that since $x \in c$ we have that 
\begin{equation} \label{eq_bdry}
\partial \Sigma_i \subset B(x,4R_0) \subset B(x,r_0/4).
\end{equation}

We first show that these choices imply that there is no non-trivial topology on intrinsic scales below $R_0$.
More precisely, we let $y \in \Sigma \cap B(x,r_0/2)$ and claim that there is a unique disk $D_y \subset \Sigma \cap B(x,r_0)$ with 
\begin{equation*}
B^\Sigma(y,R_0) \subset D_y \ \ \text{and} \ \ \partial D_y \subset \partial B^\Sigma(y,R_0).
\end{equation*}
This can be seen as follows.
By \cref{lem_gen_fund}, if there were a curve $\sigma \subset B^\Sigma(y,R_0)$,
that is non-contractible on scale $r_0$ at $x$,
we could find a simple closed, curve $\sigma' \subset B^\Sigma(y,R_0)$ also non-contractible on scale $r_0$ at $x$ with
\begin{equation*}
\length(\sigma') \leq 3R_0 < \length(c)/2 \leq \sys_{r_0}(\Sigma),
\end{equation*}
by our choice of the curve $c$, but this is impossible by the definition of the $r_0$-local systole.
We conclude that any simple closed curve contained in $B^\Sigma(y,R_0)$ admits a filling disk
contained in $\Sigma \cap B(x,r_0)$, from which the existence of $D_y$ follows.
If $\Sigma$ is not a sphere, it follows immediately that such a disk is unique.
In the case of $\Sigma$ being a sphere there are two such disks in $\Sigma$.
However, by the choice of $r_0$ and the maximum principle not both of these disks can be entirely contained in $B(x,r_0)$.

It follows from \cref{thm_chord_arc} and the convex hull property, that we can find some small $\alpha>0$ such that
\begin{equation*}
\Sigma \cap B(y,\alpha R_0) \ \text{consists of disks for any} \ y \in B(x,r_0/2).
\end{equation*}
Now choose $k \in \IN$ such that $k \alpha \geq 9$ and let $\beta_k>1$ be given by \cref{thm_chord_arc_2}.
First assume that we can find $z \in \Sigma_i$ such that 
\begin{equation} \label{eq_assum}
B^\Sigma(z,\beta_k \alpha R_0) \cap \partial \Sigma_i = \emptyset.
\end{equation}
Also assume that
\begin{equation} \label{eq_assum_2}
B^\Sigma(z,\beta_k \alpha R_0) \subset B(x,r_0/2),
\end{equation}
since the conclusion otherwise follows from the maximum principle thanks to \eqref{eq_bdry}.
Under these assumptions it follows from \cref{thm_chord_arc_2} that 
$$
\partial( (B^\Sigma(z,\beta_k \alpha R_0))_{z,9R_0} )\subset \partial B(z,9R_0),
$$
which clearly implies that 
$$
\Sigma_i \cap \partial B(x,9R_0) \neq \emptyset,
$$
since $\Sigma_i$ is connected.
Since on the other hand $\partial \Sigma_i \subset B(x,4R_0)$ we then find from the maximum principle that
$$
\Sigma_i \cap \partial B(x, r_0) \neq \emptyset.
$$

We still need to justify why we can assume \eqref{eq_assum}.
Take $l_1$ such that 
$
\beta_k \alpha R_0 \leq 16 \beta_k \alpha l_1 \leq r_0/12.
$
If with these choices \eqref{eq_assum} fails for any $z \in \Sigma_i$ we then have that
$$
\diam(\Sigma_i) \leq 32 \beta_k \alpha l_1 \leq r_0/6.
$$

Suppose first that $\Sigma_i$ is a disk. 
In this case the diameter estimate implies that $c$ is contractible on scale $r_0$ contradicting our choice of $c$.

If $\Sigma_i$ is not a disk it contains at least one non-separating curve $d$, since $\partial \Sigma_i$ is connected.
Thanks to the diameter estimate \cref{cor_gen_hom} then implies that we can find a non-separating curve $d'$ having
\begin{equation*}
\length(d') \leq 48 \beta_k \alpha R_0 < l_0,
\end{equation*}
contradicting the assumptions.
\end{proof}

Below, we solve a Plateau problem in $M \setminus \Sigma$ with boundary given by a curve $c$ as above.
In this situation, \cref{lem_good_barrier} implies that $\Sigma$ is a useful barrier.

\begin{lemma} \label{lem_no_necks}
Given $l_0>0$ there is $l_2=l_2(M,l_0)>0$ with the following property.
Let $\Sigma \subset M$ is a closed minimal surface with $\sys^h(\Sigma) \geq l_0$, then the
 $r_0$-local systole satisfies 
$$
\sys_{r_0}(\Sigma) \geq l_2.
$$
 \end{lemma}
 
 Note that this achieves two things simultaneously.
 Firstly, it shows that the systole is bounded away from zero if the homology systole is.
 Secondly, we also find that curves of controlled and very short, but potentially on a much smaller scale than the systole, can be contracted in an extrinsically controlled neighbourhood.

This corresponds to the fact, already present in the argument for \cref{lem_good_barrier}, that the proof handles two types of curves.
On the one hand, applied to homologically trivial non-contractible curves, this implies that the homology systole of a sequence $\Sigma_j$ tends to $0$
if we can show that the systole does so.
On the other hand, we will apply it to short curves bounding (large) disks in $\Sigma_j$ in order to understand the convergence of $\Sigma_j$ to a limit lamination.

\begin{figure}[ht]
	\centering
  \includegraphics[width=10cm]{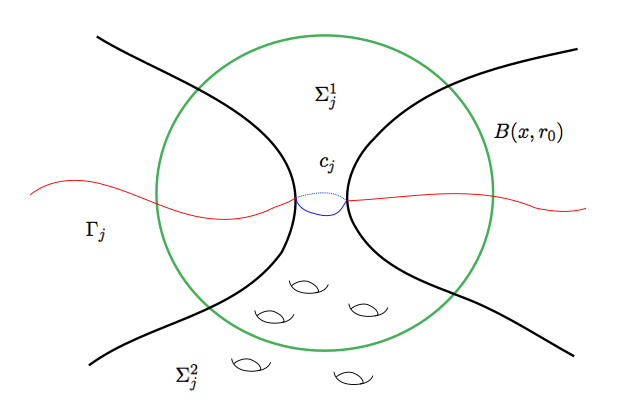}
	\caption{The construction in the proof of \cref{lem_no_necks}. The surface $\Sigma_j$ is a good barrier for the Plateau problem: Both components of $\Sigma_j \setminus c_j$ extend out of $B(x,r_0)$ by \cref{lem_good_barrier}.}
	\label{fig}
\end{figure}

\begin{proof}
Let us first consider the case of $M$ being simply connected.
Afterwards we reduce the general case to this special case.
We argue by contradiction and assume that we can find a sequence of minimal surfaces $(\Sigma_j)_{j \in \IN}$ such that
\begin{enumerate}
\item[(1)] All non-separating curves in $\Sigma_j$ have length at least $l_0$, i.e.\ $\sys^h(\Sigma_j) \geq l_0$.
\item[(2)] We have $\sys_{r_0}(\Sigma_j) \to 0$.
\end{enumerate}

Up to taking a subsequence we then find $x \in M$, radii $r_j \to 0$, and simple closed curves $c_j \in C^{\Sigma_j}(x,r_j)$ such that 
\begin{equation*}
\length(c_j) \leq 2 \sys_{r_0}(\Sigma_j) \to 0.
\end{equation*}

Since $M$ is simply connected, $\Sigma_j$ separates $M$ into two mean-convex connected components
 \begin{equation*}
 M \setminus \Sigma_j=M_j^1 \cup M_j^2.
 \end{equation*}
 Clearly, once $j$ is large enough such that $4\sys_{r_0}(\Sigma_j) \leq l_0$, we have that $c_j$ is null homologous in the closure of both of these components.

 In addition, we claim that at least one of $M^1_j$ and $M^2_j$ has the following property:
 If $\length(c_j)\leq l_1$ from \cref{lem_good_barrier}, then
 any surface $S\subset M^i_j$ with $\partial S =c_j$ satisfies
 \begin{equation} \label{int}
 S \cap \partial B(x,r_0) \neq \emptyset.
 \end{equation} 
 If this was not the case, we would find $S_j^1 \subset M_j^1 \cap B(x,r_0)$ and $S_j^2 \subset M_j^2 \cap B(x,r_0)$ such that $\partial S_j^i = c_j$.
The surface $S_j=S_j^1 \cup S_j^2 \subset B(x,r_0)$ is a closed surface and separates $B(x,r_0)$ into two connected components.
Moreover \eqref{int} does not hold for $S$, so that one of these components is contained in $B(x,r_0-\delta)$ for some small $\delta>0$.
By construction, this component contains a component of $\Sigma_j \setminus c_j$ contradicting \cref{lem_good_barrier}.

Let $M_j^1$ be the component having property \eqref{int}.
By \cite{HS79} we can find a stable minimal surface $\Gamma_j \subset M_j^1$ with $\partial \Gamma_j = c_j$ which minimizes
area among all surfaces in $M_j^1$ which have boundary $c_j$.
It follows from \eqref{int} that 
\begin{equation} \label{non_trivial}
\Gamma_j \cap \partial B(x,r_0) \neq \emptyset
\end{equation}
for $j$ sufficiently large.
Moreover, by the curvature estimates \cite{Sc83}, there is a constant $C$ such that
\begin{equation*}
\sup_j \sup_{\Gamma_j \cap (M \setminus B(x,r))} (r-r_j)^2 |A^{\Gamma_j}|^2 \leq C
\end{equation*}
for any $r>r_j$.
In particular, we can pass to a subsequence such that 
\begin{equation*}
\Gamma_j \to \cL
\end{equation*}
in $C^{0,\alpha}_{\loc}(M \setminus \{x\})$, where $\cL$ is a minimal Lipschitz lamination.
Since $\Gamma_j$ is stable, the same argument as in \cite[Lemma 4.1]{CKM} implies
\footnote{If all leaves are two-sided this follows immediately from \cite[Proposition D.3]{CKM}. The argument in \cite[Lemma 4.1]{CKM} explains how this can be assumed.}
that the lamination $\cL$ extends to a lamination $\tilde \cL$ across $\{x\}$.
By the log cut-off trick it follows that also $\tilde \cL$ has stable leaves.
From \eqref{non_trivial}, we find that there is a leaf $\bar \Gamma \subset \tilde  \cL$
with 
\begin{equation*}
\bar \Gamma \cap \partial B(x,r_0) \neq \emptyset.
\end{equation*}
In particular, $\bar \Gamma$ is non-empty.
Moreover, invoking \cite{FS} and \cite{SY83} once again, $\bar \Gamma$ is closed.
Thus, since $M$ is simply connected, we find that $\bar \Gamma$ is two-sided.
Since $M$ has positive Ricci curvature, this is a contradiction since $\bar \Gamma$ is a non-empty, two-sided, closed, stable minimal surface in $M$.

\medskip

We now consider the general case in which we can assume that $M$ is not simply connected.
We can pass to the universal covering $\pi \colon \tilde M \to M$, which is
compact by the Bonnet--Myers theorem.
In particular, there is a finite group $G$ acting freely on $M$ such that $M=\tilde M/G$.
We obtain minimal surfaces 
\begin{equation*}
\hat \Sigma_j =\pi^{-1}(\Sigma_j) \subset \tilde M.
\end{equation*}
Since $M$ has positive Ricci curvature, by the Frankel property, the surfaces $\hat \Sigma_j$ are connected.

We may assume that $r_0$ is chosen sufficiently small such that
\begin{equation*}
g(B(x,r_0)) \cap B(x,r_0)=\emptyset
\end{equation*}
for any $g \in G \setminus \{e\}$.
If there is a non-contractible curve $c_j \subset \Sigma_j \cap B(x,r_0)$, with
\begin{equation*}
\length(c_j) \leq l_0,
\end{equation*}
we may again assume that $c_j$ is chosen to have properties (1) and (2) from above.
It follows from our assumption that $c_j$ is separating.
Therefore, by \cref{lem_cov_sep}, also $\hat c_j:=\pi^{-1}(c_j)$ is separating.
Moreover, by the choice of $r_0$, and recalling $l_0 \leq r_0$, we see that $\hat c_j$ consists of $|G|$ disjoint, closed curves.
We can now argue exactly as above and
and minimize area in the correct component of $\tilde M \setminus \hat \Sigma_j$ relative to the boundary $\hat c_j$.
Finally, by \cref{lem_good_barrier}
\footnote{We apply this to $\Sigma_j$ and observe that this trivially implies \eqref{non_trivial} for $\hat \Sigma_j$.},
the limit lamination will be non-empty and we can conclude as in the first case.
\end{proof}

\begin{rem}
For curves that are non-contractible in $\Sigma \cap B(x,r)$ but contractible in $\Sigma$, it should be possible to extend \cref{lem_good_barrier} to bumpy metrics of positive scalar curvature.
In this situation one component of $\Sigma_j \setminus c_j$ is a planar domain and one can write large parts of this component as graph over $\Gamma_j$.
This can then be used to construct a non-trivial Jacobi field on $\Gamma$
\end{rem}

\begin{prop} \label{prop_ulsc}
For any $l_0>0$ there is $r_1>0$ such that for any closed minimal surface $\Sigma \subset M$ with $\sys^h(\Sigma) \geq l_0$ we have for the contractibility radius that
$c(\Sigma) \geq r_1$.
\end{prop}

\begin{proof}
If we apply \cref{lem_no_necks} to $\Sigma$ we get some $l_2>0$ such that all curves in $\Sigma$ of length at most $l_2$ are contractible in the intersection of $\Sigma$ with some mean convex ball $B(x,r_0)$.
In particular, it follows from \cref{lem_gen_fund} that any intrinsic ball $B^\Sigma(z,l_2/3)$ is contained in some disk $D_z$ with
\begin{equation*}
B^\Sigma(z,l_2/3) \subset D_z \subset \Sigma_j \cap B(z,r_0).
\end{equation*} 
The claim now easily follows with $r_1=\alpha l_2 /3$ from \cref{thm_chord_arc}, where also $\alpha>0$ is from \cref{thm_chord_arc}.
\end{proof}

\subsection{The first homology systole} 
At this stage we are in position to prove the special case $k=1$ of our main result.

\begin{proof}[Proof of \cref{thm_intrinsic} for $k=1$]
We argue by contradiction and assume that we have sequence of minimal surfaces $\Sigma_j \subset M$ with $-\chi(\Sigma_j) \to \infty$ and
\begin{equation*}
\sys^h(\Sigma_j) \geq l_0>0
\end{equation*} 
for some positive constant $l_0$.
Thanks to \cref{prop_ulsc} we find that the sequence $(\Sigma_j)$ is ulsc, i.e.\
\begin{equation} \label{claim_ulsc}
\Sigma_j \cap B(x,r_1) \ \text{consists of disks for any} \ x \in M.
\end{equation}
Clearly, after potentially decreasing $r_1$, \eqref{claim_ulsc} holds for the surfaces $\hat \Sigma_j \subset \tilde M$ as well.
Therefore, it suffices to derive a contradiction from \eqref{claim_ulsc} if $M$ is simply connected.

Thanks to \eqref{claim_ulsc} and \cite{Wh15}(see also \cite{CM5} which gives  Lipschitz curves), we can pass to a subsequence such that 
\begin{equation*}
\Sigma_j \to \cL \ \text{in} \ M \setminus \cS
\end{equation*}
outside the singular set $\cS$ which is contained in a union of $C^1$-curves.
It follows from \cref{lem_curv}, that $\cS \neq \emptyset$.
In particular, we can pick $x \in \cS$ and the associated collapsed leaf $\Gamma_x$.
Moreover, since $\Gamma_x$ is a limit leaf of $\cL$ it is stable by \cite{MPR10}.
It follows from \cref{prop_collapsed} that $\Gamma_x$ extends to a complete minimal surface $\bar \Gamma$ in $M$ and that $\cS \cap \bar \Gamma$ is discrete.
In particular, also $\bar \Gamma$ is stable and by \cite{FS} and \cite{SY83}, its universal cover is diffeomorphic to $S^2$.
Since $M$ is simply connected, it does not contain any one-sided surfaces and we conclude that $\bar \Gamma$ is a two-sided, closed, stable minimal surface in $M$.
This is clearly a contradiction, since $M$ has positive Ricci curvature.
\end{proof}
 
\section{Existence of multiple short curves}\label{main_2}
We now proceed to the proof of the general case of \cref{thm_intrinsic}.

Recall that we assume $M$ to be a closed three-manifold with positive Ricci curvature.
Assume we have a sequence of minimal surfaces $(\Sigma_j )_{j \in \IN}$ in $M$ with the following properties.
There is a natural number $k\geq 1$ and for each $j$ a set $\{c_j^1,\dots, c_j^k\}$ of simple closed curves in $\Sigma_j$ such that
\begin{itemize}
\item[(1)] $\length(c_j^i) \to 0$ for $i=1,\dots,k$ as $j \to \infty$,
\item[(2)] $\rk \langle [c_j^1], \dots, [c_j^k] \rangle = k$ in $H_1(\Sigma_j;\IZ/2\IZ)$,
\item[(3)] there is $l_0>0$ such that if a closed curve $d_j \subset \Sigma_j$ has $\length(d_j) \leq l_0$,
then $[d_j] \in \langle [c_j^1], \dots, [c_j^k] \rangle$.
\end{itemize}
Note that (3) allows for $[d_j]=0$. 

By taking a subsequence we may assume that 
$c_j^i \subset B(x_i,r_j)$ for a sequence of radii $r_j \to 0$.

We now follow the same steps that we used for the case $k=1$, but have to deal with several new difficulties.

\subsection{Additional points in the singular set}
In a first step we show that the singular points arising from the curves $c_j^i$ do not comprise the entire singular set.
This is the analogue of \cref{lem_curv}.
In contrast to \cref{lem_curv} the argument in this case relies on the assumption on the homology systole.

\begin{lemma} \label{mult_sing_set}
We have $\cS \cap M \setminus  \cup_{i=1}^k B(x_i,r_0) \neq \emptyset$ for some $r_0>0$.
\end{lemma}

\begin{proof}
Assume that $\cS \subset \{x_1,\dots,x_k\}$.
By \cref{thm_sys_ratio},  we can assume that $\area(\Sigma_j)$ is unbounded.

For simplicity, let us scale $M$ to have $|\textup{sec}|\leq 1$, and write $B_s=\cup_{i=1}^k B(x_i,s)$.
The monotonicity formula then implies
\begin{equation*}
\begin{split}
\area(\Sigma_j \cap (B_{2r_0}\setminus B_{r_0}))
&=
\area(\Sigma_j \cap B_{2r_0})-\area(\Sigma_j \cap B_{r_0})
\\
& \geq \left(\frac{4}{e^{2r_0}}-1 \right) \area(\Sigma_j \cap B_{r_0})
\\
& \geq  \area(\Sigma_j \cap B_{r_0})
\end{split}
\end{equation*}
if $r_0 \leq \log (2)/2$,
which in turn implies
\begin{equation} \label{eq_mult_area}
\begin{split}
2\area(\Sigma_j \setminus B_{r_0})
&\geq
\area(\Sigma_j \setminus B_{r_0}) 
+ 
\area(\Sigma_j \cap (B_{2r_0}\setminus B_{r_0}))
\\
&\geq
\area(\Sigma_j) \to \infty.
\end{split}
\end{equation} 
Now we can argue exactly as in \cref{lem_curv} and obtain a limit lamination $\cL \subset M \setminus \cS$.
Thanks to \eqref{eq_mult_area} we can conclude that $\cL$ has a leaf with stable universal cover.
We then use stability to extend it across the isolated singularities $\cS$ and eventually
use the log-cut off trick to conclude that this is still stable, which gives the desired contradiction.
\end{proof}

\subsection{Localized systole and contractibility radius II}
In a next step we prove that $\Sigma_j$ is ULSC off the set $\{x_1,\dots x_k\}$.

\begin{prop} \label{mult_ulsc}
Assume $(\Sigma_j)_{j \in \IN}$ is as above. 
Given $r>0$ there is $r_2=r_2(M,g,l_0,r)$ such that we have for the contractibility radius that
$c(\Sigma_j \cap (M \setminus \cup_{i=1}^k B(x_i,4r)) \geq r_2$
for $j$ sufficiently large.
\end{prop}

We want to follow the same strategy that we used to obtain \cref{prop_ulsc}, for which in turn \cref{lem_good_barrier} was the key input.
Because of the short curves $c_j^i$, we need to be more careful in how we select the scale on which we work.
Recall that in the case of \cref{lem_good_barrier} this was the smallest intrinsic scale of non-trivial topology.

In order to find the correct scale, we define functions $l_j, f_j \colon \Sigma_j \to [0,\infty)$ as follows.
For $x \in \Sigma_j$, we consider the set $C_j'$ of curves in $\Sigma_j$ given by
\begin{equation*}
C_j'(x):=\{c \colon S_1  \to \Sigma_j \, | \,  0\neq [c] \in \pi_1(\Sigma_j \cap B(x,r_0),x)  , 0=[c] \in H_1(\Sigma_j;\IZ/2\IZ)  \}.
\end{equation*}
Then the first functions is defined via
\begin{equation*}
l_j(x):=\min\{1, \inf \{\length(c) \ \lvert \ c \in C_j'(x)\}\},
\end{equation*}
and $f_j$ is a scale invariant version of (the inverse of) this, that incorporates the distance to the short curves $c_j^i$, given by
\begin{equation*}
f_j(x)= l_j(x)^{-1} \dist(x, c_j^1 \cup \dots \cup c_j^k).
\end{equation*}

\begin{proof}[Proof of \cref{mult_ulsc}]

We argue by contradiction and assume that we can find a simple closed curve $d_j \subset \Sigma_j$ such that 
\begin{equation} \label{d_short}
\length(d_j) \to 0,
\end{equation}
and
\begin{equation} \label{d_off}
d_j \subset M \setminus \bigcup_{i=1}^k B(x_i,2r),
\end{equation}
but 
\begin{equation}\label{d_nc}
d_j \ \text{is non-contractible on scale} \  r_0.
\end{equation}
If we can not find such a curve, the assertion follows from \cref{thm_chord_arc} combined with \cref{lem_gen_fund} and the convex hull property exactly as in the proof of \cref{prop_ulsc}.

Up to taking a subsequence, by \eqref{d_short} and \eqref{d_off}, we can assume that 
\begin{equation*}
d_j \to y \in M \setminus  \bigcup_{i=1}^k B(x_i,2r) .
\end{equation*}
Observe that \eqref{d_short} combined with the assumption implies that $[d_j] \in\langle [c_j^1],\dots, [c_j^k]\rangle$ for $j$ sufficiently large, which we simply assume to be the case from here on.

\smallskip

We have to distinguish the following two cases,
\begin{itemize}
\item[a)] the curve $d_j$ is non-separating, or
\item[b)] the curve $d_j$ is separating, i.e.\ $[d_j]=0$.
\end{itemize}

\smallskip

We start with case a).
In this case it follows from \cref{lem_top_sep}, that
\begin{equation*}
\Sigma_j \cap \left( M \setminus \bigcup_{i=1}^k B(x_i,r_j) \right) = \Sigma_j^1 \cup \Sigma_j^2,
\end{equation*}
where now $\Sigma_j^i$ are connected, disjoint minimal surfaces with
\begin{equation*}
\partial \Sigma_j^i \subset d_j \cup   \bigcup_{i=1}^k \partial B(x_i,r_j).
\end{equation*}
Since $d_j$ is non-separating in $\Sigma_j$,
it follows immediately that
\begin{equation} \label{eq_case_1_intersec}
\Sigma_j^i \cap \partial B(y,r_1) \neq \emptyset,
\end{equation}
holds for $i=1,2$ and for $j$ sufficiently large.
By the same arguments as in the proof of \cref{lem_no_necks} we may assume that $M$ is simply connected.
We now want to minimize area with boundary $d_j$ in $M \setminus  \cup_{i=1}^k B(x_i,r_j)$ instead of all of $M$.
In order to do so we first slightly modify the metric near $\cup_{i=1}^k \partial B(x_i,r_j)$ to obtain a mean convex domain.
Using a partition of unity we may simply choose a metric $g_j$ on $M \setminus B(x_i, r_j)$ that agrees with the original metric outside of $\cup_{i=1}^k B(x_i,2 r_j)$
and has mean convex boundary.
We can now solve the Plateau problem as before in
$(M \setminus  \cup_{i=1}^k B(x_i,r_j),g_j)$ with prescribed boundary $d_j$.
After passing to a subsequential limit we find a non-empty (thanks to \eqref{eq_case_1_intersec}) limit lamination in $(M \setminus \{x_1,\dots,x_k\},g)$. 
By stability, the limit lamination extends also across the set $\{x_1,\dots, x_k\}$ and we can argue as in the proof of \cref{lem_no_necks}.

For the remaining case b), we prove the stronger assertion that $f_j$ is uniformly bounded.
This handles case b) as follows.
If $f_j \leq C$, then
for $x \in M \setminus \cup_{i=1}^k B(x_i,2r)$, we find that
\begin{equation*}
l_j(x) \geq C^{-1} \dist(x , c_j^1 \cup \dots \cup c_j^k) \geq C^{-1}r
\end{equation*}
for $j$ sufficiently large, which contradicts \eqref{d_short}-\eqref{d_nc}.

In order to show that $f_j$ is uniformly bounded, we argue by contradiction and assume that 
\begin{equation} \label{f_unbd}
\liminf_{j \to \infty} \sup_{\Sigma_j} f_j \to \infty,
\end{equation}
which we simply assumed to be a full limit after taking another subsequence.
Note that $f_j \leq C_j$ for some constant $C_j >0$, since $\Sigma_j$ is a smooth and closed surface, therefore, we can pick
\footnote{Note that this is the standard selection procedure for such scales adapted to our situation.}
 $z_j \in \Sigma_j$, such that
\begin{equation*}
2f_j(z_j) \geq \sup_{\Sigma_j} f_j.
\end{equation*}

The assumption \eqref{f_unbd} implies that there is a loop $d_j$ based at $z_j$, that is non-contractible on scale $r_0$ 
such that
\begin{equation} \label{scale_comp}
\length(d_j) \leq o( \dist (z_j, c_j^1\cup \dots \cup c_j^k)).
\end{equation}
We can assume that any other loop $e_j$ based at $z_j$ that is non-contractible on scale $r_0$ 
has
\begin{equation} \label{scale_d}
\length(d_j)\leq 2\length(e_j).
\end{equation}

For $ z \in \Sigma_j \cap B(z_j, 2\length(d_j)))$, we find from \eqref{scale_comp} that 
\begin{equation*}
\begin{split}
\dist(z, c_j^1 \cup \dots \cup c_j^k) 
&\geq 
\dist (z_j, c_j^1 \cup \dots \cup c_j^k) - 2\length(d_j) 
\\
&\geq
\frac{1}{2} \dist (z_j, c_j^1 \cup \dots \cup c_j^k)
\end{split}
\end{equation*}
for $j$ sufficently large.
Therefore,  by the choice of $z_j$, we have 
\begin{equation} \label{min_scale}
4 l_j(z) \geq l_j(z_j)
\end{equation}
for any $z \in \Sigma_j \cap B(z_j, 2\length(d_j)))$.

Since $d_j$ is separating, we can write
\begin{equation*}
\Sigma \setminus d_j = \Sigma_j^1 \cup \Sigma_j^2
\end{equation*}
for connected minimal surfaces $\Sigma_j^i$ with boundary $d_j$.
We claim that 
\begin{equation} \label{eq_case_2_intersec}
\Sigma_j^i \cap \partial B(z_j,r_0) \neq \emptyset
\end{equation}
for $i=1,2$.

For ease of notation, we prove \eqref{eq_case_2_intersec} for $\Sigma_j^1$, the argument for $\Sigma_j^2$ is analogous.

We again distinguish two cases.
In the first case we assume that there is a simple closed curve $e_j \subset \Sigma_j^1$ with $0 \neq [e_j] \in \langle [c_j^1],\dots,[c_j^k]\rangle$.
This case
follows for homological reasons as follows.
We can pick a closed curve $g_j \subset \Sigma_j$ that intersects $e_j$ exactly once.
Since $d_j$ is separating and $e_j \subset \Sigma_j^1$, we can even choose $g_j$ with $g_j \subset \Sigma_j^1$.
But then $g_j$ has to intersect at least one of the curves $c_j^l$, which implies that 
\begin{equation*}
\Sigma_j^1 \cap B(x_l,r_j) \neq \emptyset.
\end{equation*}
Thanks to \eqref{scale_comp} this implies that
\begin{equation*}
\Sigma_j^1 \cap \partial B(z_j,\length(d_j)) \neq \emptyset.
\end{equation*}
Since $\partial \Sigma_j^1 = d_j \subset B(z_j,\length(d_j)/2)$, we then obtain \eqref{eq_case_2_intersec} from the convex hull property applied to $\Sigma_j^1$.

In the remaining case we have that
if $e_j \subset \Sigma_j^1$ is a simple closed curve with $\length(e_j) \leq l_0$, then $[e_j]=0$.
Moreover, the bound \eqref{min_scale} combined with the choice \eqref{scale_d} then implies that any simple closed curve $e_j \subset \Sigma_j^1 \cap B(z_j, 2\length(d_j))$ with $\length(e_j) \leq \length(d_j)/4 $ is contractible on scale $2 \length(d_j)$ at $z_j$.
At this point we are (for $\Sigma_j^1 \cap B(z_j, 2 \length(d_j))$) in the setting of \cref{lem_good_barrier} and can simply repeat the very same argument to
obtain
\begin{equation*}
\Sigma_j^1 \cap \partial B(z_j,2\length(d_j)) \neq \emptyset,
\end{equation*}
which in turn implies \eqref{eq_case_2_intersec} by the convex hull property using that $\partial \Sigma_j^1 \subset B(z_j, \length(d_j))$.

We can now once again argue as in the proof of \cref{lem_no_necks} and conclude the proposition.
\end{proof}

\subsection{Proof of the main result}
We now give the proof for the general case of our main result.

\begin{proof}[Proof of \cref{thm_intrinsic}]
For $\pi \colon \tilde M \to M$ the universal covering, consider
the surfaces $\hat \Sigma_j = \pi^{-1}(\Sigma_j)$
and denote
 $\mathcal{X}=\pi^{-1}(\{x_1,\dots,x_k\})$.
We can pass to a subsequential limit
\begin{equation*}
\hat \Sigma_j \to \cL  \ \ \text{in} \ \ C_{\rm loc}^{0,\alpha}(\tilde M \setminus \cS),
\end{equation*}
where clearly $\mathcal{X} \subset \cS$.
It follows from  \cref{mult_ulsc} that the surfaces are ULSC away from $\mathcal{X}$.
Moreover, thanks to \cref{mult_sing_set}, we can find a collapsed leaf $\Gamma \subset \cL$, which extends across
$\cS \setminus \mathcal{X}$ by \cref{prop_collapsed}.
Moreover, since this is stable, it also extends across the isolated points $\mathcal{X}$ to a complete, stable minimal surface,
which implies a contradiction as in the $k=1$ case.
\end{proof}

\end{document}